\documentclass[11pt,a4paper]{article}
\usepackage{authblk}
\usepackage{indentfirst}
\usepackage{graphicx}
\usepackage{epsfig}

\usepackage{latexsym}
\usepackage{amsmath}
\usepackage{amssymb}
\usepackage{amsfonts}
\usepackage{mathrsfs}
\usepackage{pifont}

\usepackage{amsthm}
\usepackage{float}
\usepackage{subfigure}
\usepackage{color}
\usepackage{multicol}

\usepackage[colorlinks,linkcolor=blue,filecolor=black,citecolor=blue]{hyperref}

\newtheorem{theorem}{Theorem}[section]
\newtheorem{lemma}[theorem]{Lemma}
\newtheorem{prop}[theorem]{Proposition}

\newtheorem{remark}{Remark}[section]

\newcommand{\mbf}{\boldsymbol}

\title{Explicit Runge approximation for Helmholtz equation with cylindrical harmonics}

\author[1]{Yu Chen}
\author[2]{Jin Cheng\thanks{Corresponding author: jcheng@fudan.edu.cn}}
\author[1]{Tingyue Li}
\affil[1]{{\small  School of Mathematics, Shanghai University of Finance and Economics, Shanghai, China}}
\affil[2]{{\small School of Mathematical Sciences, Fudan University, Shanghai, China}}
\date{}
\begin{document}

\maketitle

\begin{abstract}

Originating from the complex approximation of holomorphic functions, the Runge approximation for elliptic equations has evolved into a fundamental tool for inverse problems and even learning-based numerical methods since its proposition by Lax \cite{Lax1956} and Malgrange \cite{Malgrange1955}, with quantitative characterizations further established by  R\"uland and Salo \cite{Ruland2018}. It should be remarked here that Runge approximation is ill-posed. In numerical analysis, explicit quantitative estimates are required to characterize the dependence of the approximant’s growth on the outward continuation distance of the original solution. This paper investigates the spatial dependent quantitative Runge approximation for the Helmholtz equation using cylindrical harmonics, considering both interior and exterior boundary value problems. We explicitly derive the relevant indices for the three-circle configuration and obtain asymptotic indices for general geometric settings. The derived results provide norm estimates for the expansion coefficients, which are crucial for the implementation of regularization methods. Furthermore, the established bounds enable the construction of spectrally accurate numerical approximations for solutions to the Helmholtz equation.

\end{abstract}


\section{Introduction}

Runge approximation for elliptic equations, since being proposed by Lax \cite{Lax1956} and Malgrange \cite{Malgrange1955}, has become an increasingly significant tool in various inverse problems such as sampling method \cite{Cheng2005}, Calder\'on problem \cite{Ruland2021}, electromagnetic field control \cite{Harrach2018}, etc. It states that a solution to an elliptic equation can be approximated by a solution in a larger domain. The expense is the growth of the approximant outside the original domain. Recently, the quantitative Runge approximation proposed by R\"uland and Salo \cite{Ruland2018} reveals that, generally such growth is exponentially against the approximation error, and polynomially when the original solution can be extended outward. It has been shown to be important in the stability estimate of inverse problems \cite{Ruland2021}, even in the numerical analysis of learning based numerical method \cite{Chen2024}. The related quantitative results have also been extended to various systems, e.g. \cite{Ruland2020} \cite{Pohjola2022} \cite{Li2026}.

In numerical computations, Runge approximation allows considerable freedom in the choice basis. Solutions defined in a larger and regular domain can be used to approximate the solution in the original domain, e.g. fundamental solutions \cite{Barnett2008},  cylindrical harmonics \cite{Brubeck2022}, etc, providing adaptability and flexibility. Such an advantage is prominent in learning based methods, since the learning samples need not to be confined to the original domain, allowing a wide variety of existing data solutions to be employed.
A learning based method for Helmholtz equations is proposed in \cite{Chen2024} where the solution operator is reconstructed using fundamental solutions as learning samples. There, the quantitative Runge approximation in \cite{Ruland2018} was adopted in sample choice, regularization, and convergence analysis, to ensure interpretability and generalizability. The learning based method shows efficiency and accuracy in applications such as acoustic simulation \cite{Li2024}.

In numerical analysis, it should be noted that the growth of the approximant in Runge approximation leads to ill-posedness. 
A numerically feasible case is one in which the computational domain $\Omega$ is compactly contained in $\Omega_\delta$, the domain of existence of the solution. The corresponding bound for the approximant $u_\varepsilon$ on the larger domain $\tilde{\Omega}$ takes the form 
 \[
   % \|u|_{\partial\Omega}-u_\varepsilon|_{\partial\Omega}\|_{L^2(\partial\Omega)}\leq \varepsilon \|u\|_{H^{1/2}(\partial\Omega_\delta)}, \quad
   \| u_\varepsilon  \|_{H^{1/2}(\partial \tilde{\Omega})}\leq C\varepsilon^{-\mu }\|u\|_{L^2(\partial\Omega)}, 
   \]
which depends on some constant $\mu$ \cite{Ruland2018}. Qualitatively, $\mu$ may increase as the continuation distance $\delta$ tends to zero. However, for numerical analysis, a characterization of this dependence is desirable. Basically, the exponent is associated with some quantitative stability estimate of Cauchy problems, which, in the harmonic case, can be related to the harmonic measure \cite{Chen2022}\cite{Kabanikhin2026}. The corresponding relation in the Helmholtz case remains to be clarified.

In this work, we first derive a detailed estimate to illustrate that in the three-circle geometry, $\mu$ is explicitly related to the harmonic measure. In more general settings, we obtain that asymptotically
\(
\mu\sim \frac{C}{\delta}.
\)
In practice, cylindrical harmonics (Fourier-Bessel)
\[
\sum c_n J_n(k|x|)e^{in\theta}
\]
is a convenient choice of basis functions \cite{Gopal2019} or learning samples for interior problems. It is therefore natural to study the convergence rate, when approximated by this basis. Based on the above estimate, we can further derive bounds for the coefficients $c_n$, which are essential to apply the regularization method and the choice of regularization parameter \cite{Cheng2005}. With these results, we present an application to the numerical approximation of solutions in complex domains.

\section{Main result}
\begin{theorem}\label{runge-approx-new}
   Assume that $\Omega$ is a star-shaped domain with smooth boundary. $\Omega_\delta=\{x\in\mathbb{R}^2: dist(x,\Omega)<\delta\}$ ($\delta>0$) and $\Omega\Subset \Omega_\delta\Subset O_R $. $u$ satisfies Helmholtz equation $\Delta u+k^2u=0$ in $\Omega_\delta$ and there exists $\mu\in L^2(\partial\Omega_\delta)$ with $\|\mu\|_{L^2(\partial\Omega_\delta)}<M$ such that $u=\int_{\partial\Omega_\delta}\Phi_k(x,y)\mu(y)\mathrm{d}y$, where $\Phi_k(x,y)$ is the fundamental solution. Assume that $k^2$ is not a Dirichlet eigen value of $-\Delta$ in these domains. Then for any $\varepsilon>0$, $u|_{\Omega}$ can be approximated by a series converging in $O_R$ as
   \[
   u_\varepsilon (x) =\sum_{n\in\mathbb{Z}}c_n J_n(k|x|)e^{\mathrm{i}n\theta},
   \]
   such that
   \[
    \|u|_{\partial\Omega}-u_\varepsilon|_{\partial\Omega}\|_{L^2(\partial\Omega)}\leq \varepsilon M, \quad \| u_\varepsilon  \|_{H^{1/2}(\partial O_R)}\leq C\varepsilon^{-\frac{c}{\delta} }M, 
   \]
   where $c$ and $C$ are constants independent of $u$ and $\delta$.
\end{theorem}

%{\color{blue}
\begin{remark}
Applying [\cite{Kuttler1978}, Eq. (7)], the boundary $L^2$ error controls the interior error of the solution as
\[
\|u-u_\varepsilon\|_{L^2(\Omega)}\leqslant \frac{C_\Omega}{d}\|u-u_\varepsilon\|_{L^2(\partial\Omega)}
\]
where $d:=\min_j |k^2-E_j|/E_j$, $E_j$ are the Dirichlet eigenvalues of the domain and $C_\Omega$ is a domain-dependent constant.
%}
\end{remark}

The domains are illustrated in Fig.\ref{fig:domains-0}(a). The estimates are not trivial since $u$ is not assumed to be extended to the larger domain $O_R$. The norm of $u_\varepsilon$ is revealed to increase against $\varepsilon$ polynomially. The dependence of $\delta$ indicates that the further a solution in $\Omega$ can be extended (with larger $\delta$), the slower growth of the approximant. Moreover, the series expression facilitates numerical computation.

Similarly, for the exterior problem (Fig.\ref{fig:domains-0}(b)), we have the following quantitative approximation. 

\begin{theorem}\label{runge-approx-new-exterior}
   Assume that $\Omega$ is a shape domain with smooth boundary. $u$ is a radiating solution to the Helmholtz equation in $\mathbb{R}^2\setminus \overline{\Omega}$, and there exists $\mu\in L^2(\partial\Omega_\delta)$ with $\|\mu\|_{L^2(\partial\Omega_\delta)}<M$ such that $u=\int_{\partial\Omega_\delta}\Phi_k(x,y)\mu(y)\mathrm{d}y$. Let $\Omega_\delta\Supset \Omega$ be a smooth domain such that $dist(\partial\Omega_\delta,\Omega)\geq \delta$ and $D_\delta={\mathbb{R}^2\backslash \overline{\Omega}_\delta}$. Then for any $\varepsilon>0$, $u|_{D_\delta}$ can be approximated by a series converging in $\mathbb{R}^2\setminus \overline{O}_\rho$ with $O_\rho \Subset \Omega$ as
   \[
   u_\varepsilon (x) =\sum_{n\in\mathbb{Z}}c_n H_n^{(1)}(k|x|)e^{\mathrm{i}n\theta},
   \]
   such that
   \[
    \|u-u_\varepsilon|_{\partial \Omega_\delta}\|_{L^2(\partial\Omega_\delta)}\leq \varepsilon M, \quad \| u_\varepsilon  \|_{H^{1/2}(\partial O_\rho)}\leq C\varepsilon^{-\frac{c}{\delta} }M, 
   \]
   where $c$ and $C$ are constants independent of $u$ and $\delta$.
\end{theorem}

\begin{figure}[H]
    \centering
    \subfigure[Interior problem]{\includegraphics[height=2.0in]{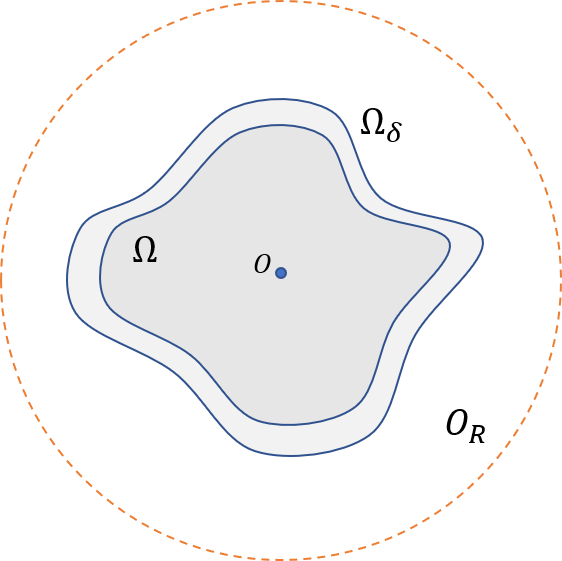}}
    \subfigure[Exterior problem]{\includegraphics[height=2.0in]{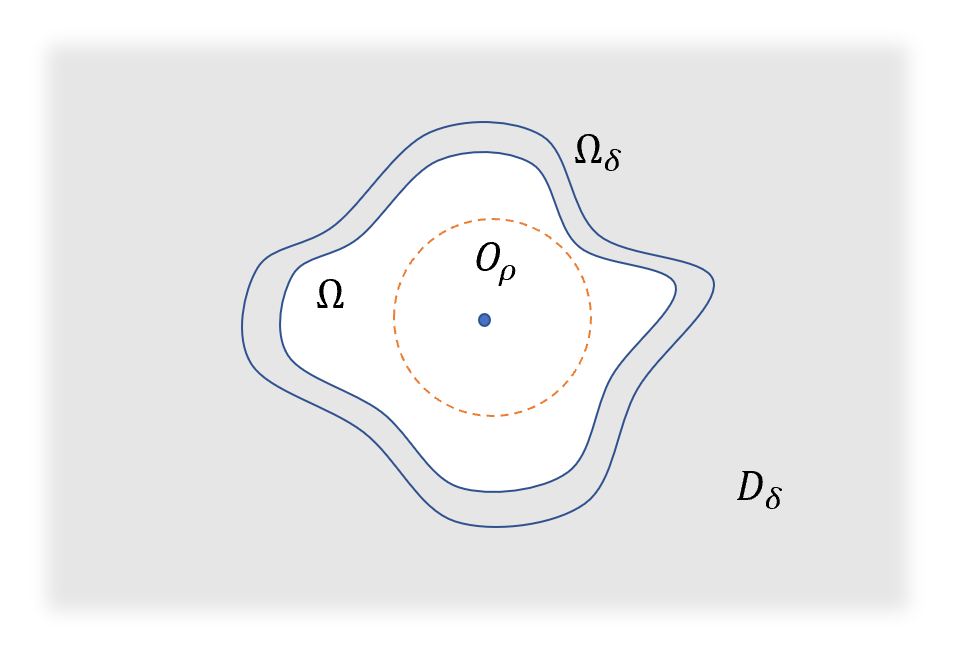}}
    \caption{Sketch of the domains}
    \label{fig:domains-0}
\end{figure}

To the authors' knowledge, existing results on quantitative Runge approximation usually consider cases with bounded domains. The present result regarding the exterior problem has potential applications in scattering problems, both theoretically and computationally.

Cauchy-Kovalevskaya theorem indicates that if the boundary is analytic with analytic data, the solution is extendable. For a domain with corners, in general the solution may not be extendable. Our future work will illustrate how to decompose the corner singularity and further extend the solution. Then, the present theory on the extendable case can be applied. 

\section{Proof of main result}

First, the following lemma gives the basic case of the exterior problem with circles.

\begin{lemma}\label{runge-ball}
        Assume that $u$ is the solution to 
        \begin{eqnarray}
		\Delta u + k^2 u & = & 0, \qquad \text{in} \quad \mathbb{R}^2 \setminus \overline{O}_R  \\
		u  & = & f(x)
		\qquad \text{on} \quad \partial O_R \\
		\lim_{r \rightarrow \infty} \sqrt{r} \left (\frac{\partial u}{\partial r} - \mathrm{i}k  u \right) & = & 0, \qquad r=|x|%\label{radiation-condition}
        \end{eqnarray}
        and that $f\in H^{1/2}(\partial O_R)$. Then there exists $\{c_n\}_{n\in\mathbb{Z}}$, such that
        \[
        u(x)=\sum_{n=-\infty}^{\infty} c_n H_n^{(1)}(kr) e^{\mathrm{i}n\theta}, \quad x\in \mathbb{R}^2\backslash \overline{O}_R.
        \]
        %where $c_n=\hat{f}_n/H_n^{(1)}(kR)$. 
        Moreover, for $\varepsilon\in (0,1)$,

        \noindent \textbf{i)} \ There exists a radiating solution $\tilde{u}$ in $\mathbb{R}^2\backslash \overline{O}_\rho $ with $\rho\in (0,R)$, such that
        \begin{equation}
        \|u-\tilde{u}|_{\partial O_R}\|_{L^2(\partial O_R)}\leq \varepsilon \|u\|_{H^{1/2}(\partial O_R)},\quad \|\tilde{u}\|_{H^{1/2}(\partial O_\rho)}\leq C e^{\left(\log\frac{R}{\rho}\right) \varepsilon^{-2}}\|u\|_{L^2(\partial O_R)}.
        \end{equation}
        
        \noindent \textbf{ii)}  If $u$ can be extended to $\mathbb{R}^2\backslash\overline{O}_\sigma$, $0<\rho<\sigma<R$, then there exists a radiating solution $\tilde{u}$ in $\mathbb{R}^2\backslash\overline{O}_\rho$, such that
        \begin{equation}
        \|u-\tilde{u}|_{\partial O_R}\|_{L^2(\partial O_R)}\leq \varepsilon \|u\|_{H^{1/2}(\partial O_\sigma)},\quad \|\tilde{u}\|_{H^{1/2}(\partial O_\rho)}\leq C \varepsilon^{-c{\frac{\log\frac{\rho}{R}}{\log \frac{\sigma}{R}}}}\|u\|_{L^2(\partial O_R)}.
        \end{equation}

    \end{lemma}

\begin{proof}
%Without loss of generality, we assume that $R\leq 1$ in the following. 
Denote the trace $g=\gamma u|_{\partial O_R}$ and $f(\theta), \,\theta\in[0,2\pi]$ the parameterization of $g$. Considering that $\{\mathrm{e}^{\mathrm{i}n\theta}\}_{n\in\mathbb{Z}}$ forms a complete basis for $L^2(0,2\pi)$, there is \[f(\theta)=\sum_{m\in \mathbb{Z}}\hat{f}_m e^{\mathrm{i}m\theta},\quad \theta\in [0,2\pi]; \quad   \hat{f}_m=\frac1{2\pi}\int_0^{2\pi}f(\theta)e^{-in\theta}\,\mathrm{d}\theta.\]
%Since $f\in H^{1/2}(\partial O_R)$,  
%\[
%2\pi N \sum_{|m|>N}|\hat{f}_m|^2\leq    2\pi\sum_{|m|>N}R\sqrt{1+\frac{m^2}{R^2}}|\hat{f}_m|^2\leq \|f\|^2_{H^{1/2}(\partial O_R)}
%\]
The Hankel functions $H_m^{(1)}(kr)$ naturally satisfy the radiation condition as $r\rightarrow +\infty$. Meanwhile, $H^{(1)}_m(kr)\neq 0, r>0$. Therefore, due to the uniqueness of the exterior scattering problem,
\[
u(\mbf{x})=\sum_{m\in\mathbb{Z}}C_mH_m^{(1)}(kr)e^{\mathrm{i}m\theta}, \quad C_m=\frac{\hat{f}_m}{H_m^{(1)}(kR)},\quad \mbf{x}\in \mathbb{R}^2\setminus \overline{O}_R.
\]
 Define 
        \[
        \tilde{u}_N(r,\theta )=\sum_{|m|\leq N} \frac{\hat{f}_m}{H_m^{(1)}(k R)}H_m^{(1)}(kr)e^{\mathrm{i}m\theta}.
        \]
         For $\varepsilon>0$, take $N=\left[ \frac{1}{\varepsilon^2}\right]+1$, then
        \begin{align*}
        \|u-\tilde{u}_N|_{\partial O_R}\|^2_{L^2(\partial O_R)} & =2\pi \sum_{|m|>N}\frac{1}{R\sqrt{1+\frac{m^2}{R^2}}}R\sqrt{1+\frac{m^2}{R^2}}|\hat{f}_m|^2  \\
        & \leq \frac{2\pi R}{N}\sum_{m\in\mathbb{Z}}\sqrt{1+\frac{m^2}{R^2}}|\hat{f}_m|^2\leq \varepsilon^2 \|g\|^2_{H^{1/2}(\partial O_R)},
        \end{align*}
        i.e., $\|u-\tilde{u}_N|_{\partial O_R}\|_{L^2(\partial O_R)}\leq \varepsilon \|g\|_{H^{1/2}(\partial O_R)}$. 
      Then, for the norm of $\tilde{u}$ on $\partial O_\rho$ with $\rho< R$, 
        \[
        \sum_{|m|\leq N}\sqrt{1+\frac{m^2}{\rho^2}}\left|\hat{f}_m\frac{H^{(1)}_m(k\rho)}{H^{(1)}_m(kR)}\right|^2\leq C_1 N \left(\frac{R}{\rho}\right)^{2N}\sum_{|m|\leq N}|\hat{f}_m|^2,
        \]
        where we have used $H_n^{(1)}(t)=\frac{2^n(n-1)!}{\pi \mathrm{i} t^n}(1+O(\frac{1}{n})), n\rightarrow \infty $ uniformly on compact subsets of $(0,+\infty)$.
        As a result,
        \[
        \|\tilde{u}_N\|_{H^{1/2}(\partial O_\rho)}\leq C\sqrt{N}e^{N\log \frac{R}{\rho}}\|g\|_{L^2(\partial O_R)}\leq C e^{c\left(\log\frac{R}{\rho}\right) \varepsilon^{-2}}\|g\|_{L^2(\partial O_R)}.
        \]

        For case ii), when $u$ is a solution to the exterior problem in $O_\sigma$, $0<\rho<\sigma<R$, we have proved above that there exist coefficients $\{C_m\}_{m\in \mathbb{Z}}$, such that
        \[
        u(x)=\sum_{m\in\mathbb{Z}}C_mH_m^{(1)}(kr)e^{\mathrm{i}m\theta},\quad x\in O_\sigma.
        \]
        Let $g=\gamma u|_{\partial O_\sigma}$ and $f(\theta)=g(\mbf{x}(\theta))$ its parameterization, then $\hat{f}(m)=C_mH_m^{(1)}(k\sigma)$. Define the finite sum
        \[
        \tilde{u}_N(x)=\sum_{|m|\leqslant N}C_m H_m^{(1)}(kr)e^{\mathrm{i}m\theta}, \quad C_m=\frac{\hat{f}(m)}{H_m^{(1)}(k\sigma)},
        \]
        which satisfies the Helmholtz equation in $\mathbb{R}^2\backslash\overline{O}_\rho$ and the radiation condition. Similarly, there exists constant $C>0$ such that
        \begin{align*}
        \|u-\tilde{u}_N|_{\partial O_R}\|_{L^2(\partial O_R)}^2 & =2\pi \sum_{|m|>N}\frac{1}{R\sqrt{1+\frac{m^2}{R^2}}}R\sqrt{1+\frac{m^2}{R^2}}\left|\frac{\hat{f}(m)}{H_m^{(1)}(k\sigma)}H_m^{(1)}(kR)\right|^2\\
        &\leq  \frac{C}{N}\left(\frac{\sigma}{R}\right)^{2N}\sum_{m\in \mathbb{Z}}\sqrt{1+\frac{m^2}{R^2}}|\hat{f}_m|^2%\leq \varepsilon^2\|g\|^2_{H^{1/2}(\partial O_R)}
        \end{align*}
        For the constant $C$ and any given $0<\varepsilon_0<1$, there exists a constant $c_1$ such that $\varepsilon^{c_1}\leq \frac{\varepsilon}{C}$, $\forall \varepsilon\in (0,\varepsilon_0)$. Then, for sufficiently large $N$ that $\left(\frac{\sigma}{R}\right)^N\leq \varepsilon^{c_1}\leq \frac{\varepsilon}{C}$, e.g., $N=[c_1\log \varepsilon/\log\frac{\sigma}{R}]+1$. %,  where $C, c_1>1$ are constants.
        This gives
        \[
        \|u-\tilde{u}_N|_{\partial O_R}\|_{L^2(\partial O_R)}\leq \varepsilon\|g\|_{H^{1/2}(\partial O_R)}.
        \]
        Then, for the increase of $\|\tilde{u}\|_{H^{1/2}(\partial O_\rho)}$ one has,
        \begin{align*}
        &\sum_{|m|\leqslant N}\sqrt{1+\frac{m^2}{\rho^2}}\left|\hat{f}_m\frac{H^{(1)}_m(k\rho)}{H^{(1)}_m(k\sigma)}\right|^2 =\sum_{|m|\leqslant N}\sqrt{1+\frac{m^2}{\rho^2}}\left|\hat{f}_m\frac{H^{(1)}_m(k\rho)}{H^{(1)}_m(kR)}\frac{H^{(1)}_m(kR)}{H^{(1)}_m(k\sigma)}\right|^2\\
        & \lesssim N\left(\frac{R}{\rho}\right)^{2N}\sum_{|m|\leqslant N}|\hat{f}_m|^2\left|\frac{H^{(1)}_m(kR)}{H^{(1)}_m(k\sigma)}\right|^2
        \lesssim e^{2c_2N\log\left(\frac{R}{\rho}\right)}\|u\|^2_{L^2(\partial O_R)}\\
        &\lesssim \exp\left( c\log\varepsilon\cdot\frac{\log\frac{R}{\rho}}{\log \frac{\sigma}{R}}\right)\|u\|^2_{L^2(\partial O_R)}\lesssim \varepsilon^{-c\frac{\log\frac{\rho}{R}}{\log \frac{\sigma}{R}}}\|u\|^2_{L^2(\partial O_R)}.
        \end{align*}
       % where we have used
       % \[
       % \exp\left(c\log C\cdot \frac{\log\frac{\rho}{R}}{\log \frac{\sigma}{R}}\right)\leq 1
       % \]
       % with constants $C>1, c>0$.
     This completes the proof.   
         \end{proof}

%Remark: The proof is similar to the interior problem case, by replacing the Hankel functions with Bessel ones and applying the corresponding asymptotic behaviors. Similarly, let $\delta =R-\sigma $, then there exists constant $C$ such that
%\[
%0<\frac{\log \frac{\rho}{R}}{\log\frac{\sigma}{R}}\leq \frac{C}{\delta}.
%\]

We remark that the function $\phi(\sigma,\theta)=\frac{\log \frac{\sigma}{R}}{\log\frac{\rho}{R}}$ in the exponential is the same as the harmonic measure in the annulus, satisfying
\begin{align*}
    \Delta \phi=0,&\quad \text{in}\quad O_R\setminus\overline{O}_\rho,\\
     \phi=1,&\quad \text{on} \quad \partial O_\rho,\\
     \phi = 0,& \quad \text{on}\quad \partial O_R.
\end{align*}
 
For the interior problem case, we have the corresponding estimate. 

   \begin{prop}\label{runge-ball0}
        Assume that $u$ satisfies 
        \begin{eqnarray}
		\Delta u + k^2 u & = & 0, \qquad \text{in} \quad O_R,  \\
		u  & = & f(x)
		\qquad \text{on} \quad \partial O_R, 
%\label{radiation-condition}
        \end{eqnarray}
        and that $f\in H^{1/2}(\partial O_R)$. Assume that $k^2$ is not a Dirichlet eigen value for $-\Delta$ on $O_R$. Then there exists $\{c_n\}_{n\in\mathbb{Z}}$, s.t. 
        \[
        u(x)=\sum_{n=-\infty}^{\infty} c_n J_n(kr) e^{\mathrm{i}n\theta}, \quad x\in O_R.
        \]
        %where $c_n=\hat{f}_n/H_n^{(1)}(kR)$. 
        Moreover, for $\varepsilon\in (0,1)$,

        \noindent \textbf{i)} \ There exists $\tilde{u}$ that satisfies the problem in $O_\rho $ with $\rho>R$, such that
        \begin{equation}
        \|u-\tilde{u}|_{\partial O_R}\|_{L^2(\partial O_R)}\leq \varepsilon \|u\|_{H^{1/2}(\partial O_R)},\quad \|\tilde{u}\|_{H^{1/2}(\partial O_\rho)}\leq C e^{\left(\log\frac{\rho}{R}\right) \varepsilon^{-2}}\|u\|_{L^2(\partial O_R)}.
        \end{equation}
        
        \noindent \textbf{ii)}  If $u$ can be extended to $O_\sigma$, $R<\sigma<\rho$, then there exists $\tilde{u}$ satisfying the equation in $O_\rho$, such that
        \begin{equation}
        \|u-\tilde{u}|_{\partial O_R}\|_{L^2(\partial O_R)}\leq \varepsilon \|u\|_{H^{1/2}(\partial O_\sigma)},\quad \|\tilde{u}\|_{H^{1/2}(\partial O_\rho)}\leq C \varepsilon^{-{\frac{\log\frac{\rho}{R}}{\log \frac{\sigma}{R}}}}\|u\|_{L^2(\partial O_R)}.
        \end{equation}

    \end{prop}

The proof is similar to the previous case and can be found in Appendix. 
Note that let $\delta =\sigma -R $, then there exists a constant $C$ such that
\[
0<\frac{\log \frac{\rho}{R}}{\log\frac{\sigma}{R}}=\frac{\log \frac{\rho}{R}}{\log\left(1+\frac{\delta}{R}\right)}\leq \frac{C}{\delta}.
\]

The following lemma gives the quantitative approximation of a solution outside a disk by the one outside another disk.
\begin{lemma}\label{move-forward}
Assume that $y_0,\cdots, y_K$ are uniformly spaced points along the straight line segment $L$ connecting $y_0$ and $y_K$, such that $|y_{k}-y_{k-1}|=h/3$, $k=1,\cdots, K$. $U=\{x\in \mathbb{R}^2, |x-y|<h, \forall y\in L\}$. 
Assume that $u$ satisfies Helmholtz equation in $\mathbb{R}^2\setminus \overline{O}(y_0,h/3)$ and Sommerfeld radiation condition at $\infty$ and
\[
\|u\|_{H^{1/2}(\partial O(y_0,h/3))}\leq M.
\]
Then $u$ can be approximated by
\[
u_{K}=\sum_{n\in\mathbb{Z}} c_{n} H_n^{(1)}(k|x-y_{K}|) e^{\mathrm{i}n\theta_{y_{K}}},
\]
which converges in $\mathbb{R}^2\setminus O(y_K,h/3)$, and such that
\[
\|u_K-u\|_{L^2(\partial U)}\leq \varepsilon M,\quad \|u_K\|_{L^2(\partial O(y_K,h/3))}\leq \varepsilon^{-\frac{C}{h^K}}M,
\]
for some constant $C$.
\end{lemma}

\begin{proof}Since $u$ satisfies the equation in $(O(y_{1},2h/3))^c \subset  (O(y_{0},h/3))^c$  (Fig.\ref{fig:placeholder2}),  we apply Lemma \ref{runge-ball} with radii $R=h$, $\sigma=2h/3 $, $\rho=h/3$ and the center $y_1$, there exist $\{c_{1,n}\}_{n\in\mathbb{Z}}$  such that $u$ can be approximated by
\[
u_{1}(x)=\sum_{n\in\mathbb{Z}}c_{1,n}H^{(1)}_n(k|x-y_{1}|)e^{\mathrm{i}n\theta_{y_{1}}},
\]
%in $\mathbb{R}^2\backslash \overline{O}(y_{i,2},h)$ uniformly with error $\varepsilon$ and bounded of order $\varepsilon^{-\mu}$ at $\partial O(y_{i,2},\delta)$. ( where we have assumed that $\|u\|_{H^{1/2}(\partial O(y_i,\delta))})\leq 1$ for normalization ). By the same argument,  $u_{i,2} $  can by approximated by series solution $u_{i,3}$ in $\mathbb{R}^2\backslash \overline{O}(y_{i,3},3\delta)\subset  \mathbb{R}^2\backslash \overline{O}(y_{i,2},2\delta)$ 
%
%with error $\varepsilon$.
with error
\[
\|u_{1}-u\|_{L^2(\partial O(y_{1},h))}\leq \varepsilon\|u\|_{H^{1/2}(\partial O(y_{0},h/3))}=:\varepsilon M
\]
while there exists a constant $\alpha\sim \frac{C}{h}$ such that 
\[
\|u_1\|_{H^{1/2}(\partial O(y_1,h/3))}\lesssim \varepsilon^{-\alpha}\|u\|_{L^2(\partial O(y_{1},h))}\lesssim \varepsilon^{-\alpha}\|u\|_{L^2(\partial O(y_{0},h/3))}\lesssim\varepsilon^{-\alpha}M.
\]

\begin{figure}[H]
        \centering
        \includegraphics[width=0.5\linewidth]{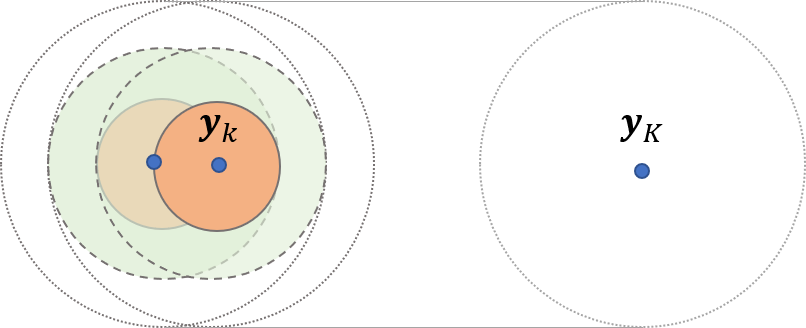}
        \caption{Sketch of the circle chain in the proof of Lemma 3.3.}
        \label{fig:placeholder2}
    \end{figure}

Next, since $u_1$ satisfies the equation in $(O(y_{2},2h/3))^c \subset  (O(y_{1},h/3))^c$,  we apply Lemma \ref{runge-ball} again with $R=h$, $\sigma=2h/3 $, $\rho=h/3$ and the center $y_2$, there exist $\{c_{2,n}\}_{n\in\mathbb{Z}}$  such that  $u$ can be approximated by
\[
u_{2}(x)=\sum_{n\in\mathbb{Z}}c_{2,n}H^{(1)}_n(k|x-y_{2})|)e^{\mathrm{i}n\theta_{y_{2}}},
\]
%in $\mathbb{R}^2\backslash \overline{O}(y_{i,2},h)$ uniformly with error $\varepsilon$ and bounded of order $\varepsilon^{-\mu}$ at $\partial O(y_{i,2},\delta)$. ( where we have assumed that $\|u\|_{H^{1/2}(\partial O(y_i,\delta))})\leq 1$ for normalization ). By the same argument,  $u_{i,2} $  can by approximated by series solution $u_{i,3}$ in $\mathbb{R}^2\backslash \overline{O}(y_{i,3},3\delta)\subset  \mathbb{R}^2\backslash \overline{O}(y_{i,2},2\delta)$ 
%
with relative error $\varepsilon^{1+\alpha_1}$, $\alpha_1=\alpha$. Then 
\[
\|u_{2}-u_{1}\|_{L^2(\partial O(y_{2},h))}\leq \varepsilon^{1+\alpha_1}\|u_{1}\|_{H^{1/2}(\partial O(y_{1},h/3))}\leq C\varepsilon^{-\alpha_1}\varepsilon^{1+\alpha_1}M=C\varepsilon M.
\]
This yields that
\begin{align*}
\|u-u_2\|_{L^2(\partial U)} & \leq\|u-u_1\|_{L^2(\partial U)}+\|u_1-u_2\|_{L^2(\partial U)}\\
& \lesssim \|u-u_1\|_{L^2(\partial O(y_1,h))}+\|u_1-u_2\|_{L^2(\partial O(y_2,h))}\lesssim \varepsilon M
\end{align*}
Meanwhile, the bound increase is 
\[
\|u_{2}\|_{\partial O(y_{2},h/3)}\leq C (\varepsilon^{1+\alpha_1})^{-\alpha}\|u_{1}\|_{L^2(\partial O(y_{1},2h/3))}= C \varepsilon^{-(1+\alpha_1)\alpha-\alpha_1}M=: C\varepsilon^{-\alpha_2}M.
\]

By a similar argument, $u_{2}$ can be approximated by $u_{3}$ in $(O(y_3,h))^c$ with error order $\varepsilon^{1+\alpha_2}$ while bounded of order 
\[\varepsilon^{-\alpha (1+\alpha_2)}\|u_{2}\|_{L^2(\partial O(y_2,2h/3))}\leq \varepsilon^{-\alpha (1+\alpha_2)-\alpha_2}M=:\varepsilon^{-\alpha_3}M.
\]
We have $\alpha_i=\alpha (1+\alpha_{i-1})+\alpha_{i-1}=\alpha+(1+\alpha)\alpha_{i-1}$.

Finally, $u_{K-1}$, and subsequently $u$, can be approximated by
\[
u_{K}=\sum_{n\in\mathbb{Z}} c_{K,n} H_n^{(1)}(k|x-y_{K}|) e^{\mathrm{i}n\theta_{y_{K}}},
\]
in $\mathbb{R}^2\setminus\overline{U}$ with the same order of error $\varepsilon M$ and bounded by $\varepsilon^{-\alpha_K}M$ for some constant $C$ and $\alpha_K\sim \frac{C}{h^K}$ depends on $h$.
\end{proof}

\

\begin{remark}
    By the above proof, the conclusion also holds for the case that $y_0,\cdots, y_K$ form a poly-line.
\end{remark}

%\[
%\int_0^L\int_0^L \ln^2 |x-y|\mathrm{d}x\mathrm{d}y
%\]

Now we move to the proof of the main results. 
As can be seen in the following, the main ideas for interior and exterior problems are the same. So we present details for the interior case.

\paragraph{Proof of Theorem 2.1}

\begin{figure}[H]
    \centering
    \includegraphics[width=0.5\linewidth]{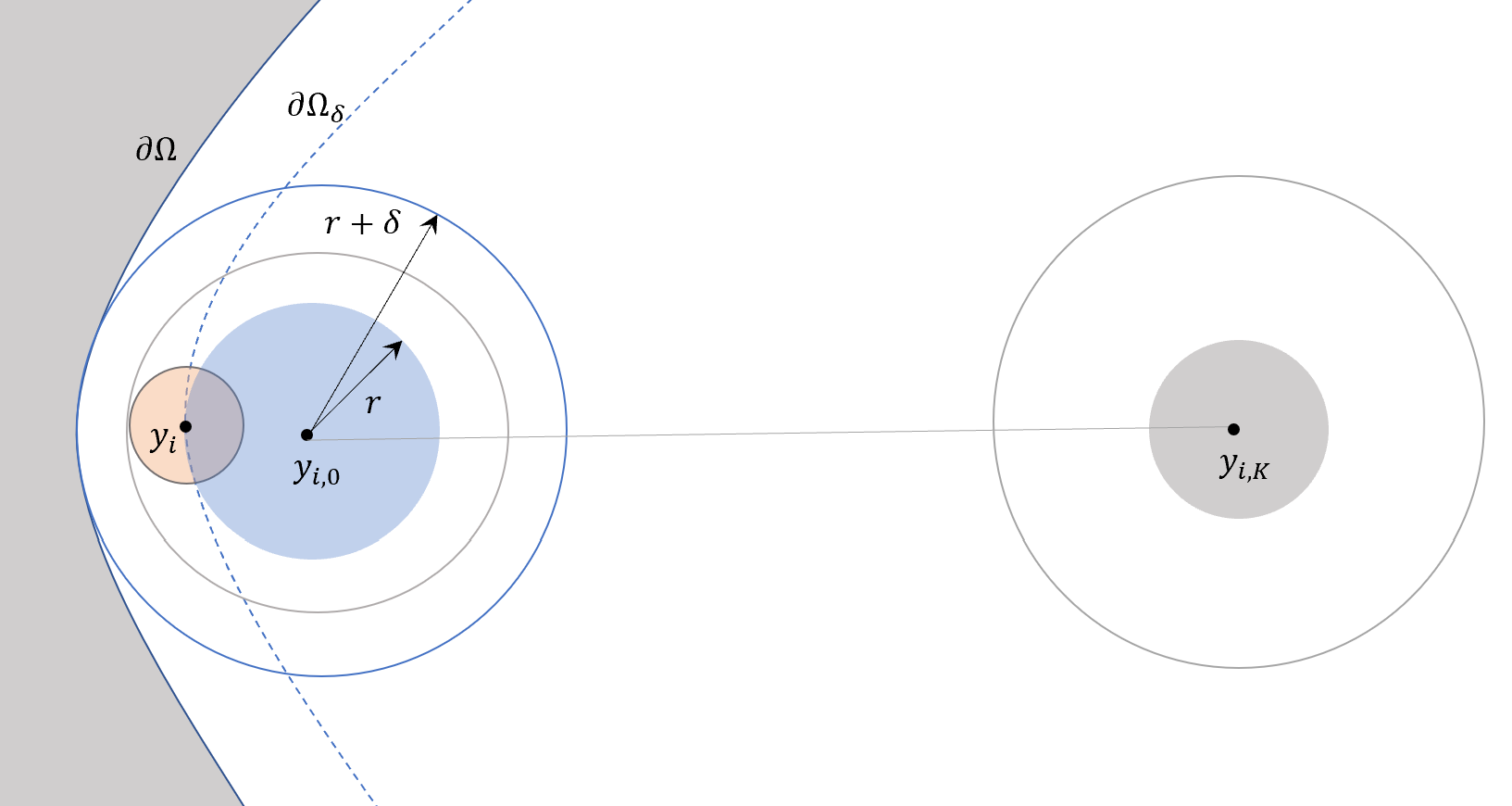}
    \caption{Sketch of domains in the proof of Theorem 2.1.}
    \label{sketch6}
\end{figure}

 \begin{proof}

By partitioning the smooth boundary curve \(\partial \Omega_
\delta\) into sub-arcs of length at most \(\delta/2\), we can have that $\partial\Omega_{\delta} =\cup_{i=1}^N l_i$ with $N=O(1/\delta)$ and $l_i\cap l_j=\emptyset $ for $i\neq j$. Denote $y_i\in \partial\Omega_\delta$ the arc length center of each $l_i$, then the sub-arc $l_i\Subset O(y_i,\delta/2)$ .  By smoothness of $\partial \Omega_{\delta}$, there exist $r>0$, and $y_{i,0}$ such that $O(y_{i,0},r)$ is an exterior ball of $\Omega_{\delta}$ tangent to $\partial\Omega_{\delta}$ at $y_i$, and such $r$ also uniformly holds for $i=1,\cdots, N$. Let $h=r+\delta$, then $O(y_{i,0},h)\subset \mathbb{R}^2\setminus \Omega$ according to the assumption on  $\Omega_{\delta}$ and $\Omega$.  

Consider
\[
u_i(x)=\int_{l_i}H_0^{(1)}(k|x-y|)\mu (y)\mathrm{d}y.
\]
It satisfies Helmholtz equation in $\mathbb{R}^2\backslash \overline{O}(y_{i,0},h-\delta/2)$. The geometries are as sketched in Fig. \ref{sketch6}.  Therefore, applying Lemma \ref{runge-ball} with $R=h$, $\sigma =h-\delta/2 $, $\rho=h/3$ (the conclusions in Lemma \ref{runge-ball} still holds for $\rho>\sigma$ or even $\rho=R$), $u_i(x)$ can be approximated by 
\[
u_{i,0}(x)=\sum_{n\in\mathbb{Z}}c^{i,0}_nH^{(1)}_n(k|x-y_i|)e^{\mathrm{i}n\theta _{y_{i,0}}},\quad x\in \mathbb{R}^2\backslash \overline{O}(y_{i,0},h/3),
\]
in $D_\delta=\mathbb{R}^2\backslash \overline{O}(y_{i,0},h)$ such that
\[
\|u_{i,0}-u_{i}\|_{L^2(\partial O(y_{i,0},h ))}\leq \varepsilon\|u_{i}\|_{H^{1/2}(\partial O(y_{i,0},h-\delta/2))} .%\leq C\varepsilon\|u\|_{H^{1/2}(\partial\Omega_\delta)}=:\varepsilon M.
\]
By Lemma \ref{estimate-u12}, %in which
%\[
%\|u_{i}\|_{H^{1/2}(\partial O(y_{i,0},h-\delta/4))}\leq C(\delta)\|\mu\|_{L^2(l_i)}\leq C(\delta)\|\mu\|_{\partial\Omega_{\delta/2}}\leq C(\delta)\|u\|_{H^{1/2}(\partial\Omega_\delta)}=:M.
%\]
\[
\|u_{i}\|_{H^{1/2}}(\partial O(y_{i,0},h-\delta/2))\leq C\|u_i\|_{H^{1/2}(\partial O(y_i,\delta/2))}\leq C\sqrt{\delta}\|\mu\|_{L^2(l_i)}=:\sqrt{\delta}M_{i,0}. 
\]

\noindent Therefore,
\[
\|u_{i,0}-u_{i}\|_{L^2(\partial O(y_{i,0},h ))}\leq \varepsilon \sqrt{\delta} M_{i,0}.
\]
Meanwhile, the bound increase is 
\[
\|u_{i,0}\|_{H^{1/2}(\partial O(y_{i,0},h/3))}\leq C \varepsilon^{-\alpha_0 }\|u_{i}\|_{L^2(\partial O(y_{i,0},h-\delta/2))}\leq C \varepsilon^{-C/\delta }\sqrt{\delta}M_{i,0}=:M_{i,1}.
\]
By the star-shaped geometry, there exists a line segment $L$ connecting $y_0$ and $y_K$, partitioned by $y_{i,0},y_{i,2},\cdots ,y_{i,K}$ so that $|y_{i,k}-y_{i,k-1}|=h/3$ and $y_{i,K-2}\in \mathbb{R}^2\setminus O(0,\rho_0 )$, where $\rho_0=\max_{x\in \partial\Omega_\delta}|x|$, and meanwhile  $O(y_{i,k},h)\subset \Omega^c$ (star-shaped).

According to Lemma \ref{move-forward}, there exists
\[
u_{i,K}=\sum_{n\in\mathbb{Z}} c_{n} H_n^{(1)}(k|x-y_{i,K}|) e^{\mathrm{i}n\theta_{y_{i,K}}},
\]
such that with $\alpha_0=\frac{C}{\delta}$ there is
\[
\|u_{i,K}-u_{i,0}\|_{L^2(\partial \Omega)}\leq \varepsilon^{1+\alpha_0} M_{i,1}=C\varepsilon \sqrt{\delta}M_{i,0},\]
and
\[
\|u_{i,K}\|_{L^2(\partial O(y_K,h/3))}\leq (\varepsilon^{1+\alpha_0})^{-\frac{C}{h^K}}M_{i,1}\leq \varepsilon^{-\alpha_0 \frac{C'}{r^K}-\alpha_0}\sqrt{\delta}M_{i,0}\leq \varepsilon^{-\frac{C(r)}{\delta}}\sqrt{\delta}M_{i,0},
\]
for some constant $C$. Notice that the exterior ball constant $r$ is dominantly affected by the curvature of $\partial \Omega$ so that $\delta$ is a higher order factor as $\delta\to 0$. %and  (the exterior ball constant for $\Omega_\delta$ is larger than $\Omega$??).

On the other hand, $u_{i,K}$ satisfies Helmholtz equation in $O(0,\rho_0+h/3) \subset (O(y_{i,K}, h/3))^c$. Then, by completeness of $\{e^{\mathrm{i}n\theta}\}_{n\in\mathbb{Z}}$ in $L^2(0,2\pi)$, $u_{i,K}$ can be expressed by
%one can apply the quantitative Runge approximation for interior problem, i.e., Lemma \ref{runge-ball0}, and yield that $u_{i,K}$ can be approximated by 
\[
\tilde{u}_i=\sum_{n\in\mathbb{Z}}c^i_n J_n(k|x|)e^{\mathrm{i}n\theta}.
\]
Consequently, $u$ can be approximated by
\[
\tilde{u}=\sum_{i=1}^N\tilde{u}_i=\sum_{n\in\mathbb{Z}}c_nJ_n(k|x|)e^{\mathrm{i}n\theta},
\]
in $\Omega$ with error 
 \begin{align*}
 \|u-\tilde{u}\|_{L^2(\partial \Omega)}%& = \int_{\partial\Omega} \left(\sum_{i=1}^N(u_i-\tilde{u}_i)\right)^2\leq \int_{\Omega}N\sum_{i=1}^N(u_i-\tilde{u}_i)^2\\%=N\sum_{i=1}^N\int_{\partial\Omega}(u_i-\tilde{u}_i)^2\\
 &\leq  \sum_{i=1}^N\|u_i-\tilde{u}_i\|_{L^2(\partial \Omega)} \leq \sqrt{N}\left(\sum_{i=1}^N\|u_i-\tilde{u}_i\|_{L^2(\partial \Omega)}^2\right)^{1/2} \\
 &\leq \sqrt{N}\left(\sum_{i=1}^N\varepsilon^2\delta M_{i,0}^2\right)^{1/2}=\frac{C}{\sqrt{\delta}}\varepsilon \left(\delta\sum_{i=1}^N\|
 \mu\|_{L^2(l_i)}^2\right)^{1/2}\\
 &= C\varepsilon M.%\frac{C}{\sqrt{\delta}}\varepsilon \sqrt{\delta}\|\mu\|_{L^2(\partial\Omega_{\delta/2})} %\leq\frac{C}{\sqrt{\delta}}\varepsilon \|u\|_{H^{1/2}(\partial\Omega_\delta)} 
 \end{align*}
and bounded on $\partial O_{\rho_0}$ by 
 \begin{align*}
\|\tilde{u}\|_{L^2(\partial O_{\rho_0})}& = \| \sum_{i=1}^N\tilde{u}_{i}\|_{L^2(\partial O_{\rho_0})}\leq \sqrt{N}\left(\sum_{i=1}^N\|\tilde{u}_i\|^2_{L^2\partial O(y_{i,K},h/3)}\right)^{1/2} \\
&\leq \frac{C}{\sqrt{\delta}}\varepsilon^{-\frac{C}{\delta}}\left(\delta\sum_{i=1}^N M_{i,0}^2\right)^{1/2}=C\varepsilon^{-\frac{C}{\delta}}M.%\|u\|_{H^{1/2}(\partial\Omega_\delta)}.
 \end{align*}
 %$\varepsilon M$ in $O(0,\rho_0)$, and bounded by $\varepsilon^{-\frac{c}{\delta}}M$ on $\partial O_{\rho_0}$  ($\Omega\subset O_{\rho_0}$). 

We have obtained the quantitative approximation by solutions on $O_{\rho_0}$. For those on $O_R$ with $R>\rho_0$, one can further applying Proposition \ref{runge-ball0} for $\tilde{u}$ with $\delta$-independent exponents involved and thus complete the proof.
\end{proof}

\;

\begin{remark}
 The case of the exterior problem (Theorem \ref{runge-approx-new-exterior}) can be proved in an almost similar way, while choosing the path $y_{i,0},\cdots,y_{i,K}\in \Omega$ in the inward direction.
\end{remark}

\begin{remark}
    For the case of non-star shape with simply connected domain, the proved can be adapted by using a poly-line connect $y_{i,0}$ to $y_{i,K}$. 
\end{remark}

%By applying Lemma 1, the proof of Theorem 1 can be directly extend to this case.

\section{Application in numerical computation}

In the following, we illustrate that with the Runge approximation, a solution that is continuable can be approximated by cylindrical harmonics with spectral accuracy.

\begin{theorem}\label{runge-part}
	With the assumptions in Theorem \ref{runge-approx-new-exterior}, further assume that $ O_\rho\Subset O_R\Subset\Omega\Subset \Omega_\delta $. Then, for sufficiently large $N$, there exist  $\{c_n\}_{n=-N}^N$ such that 
	\begin{equation}\label{num-series}
	u_N(x)=\sum_{n=-N}^N c_n \frac{H_n^{(1)} (kr)}{H_n^{(1)} (k\rho)}e^{\mathrm{i}n\theta}
	\end{equation}
approximates $u$ with  \[
	\| u-u_N \|_{L^2(\partial\Omega_\delta)}\leq  \left(\frac{\rho}{R}\right)^{N\cdot \frac{1}{1+\beta}} M,\; \text{and}\quad \left(\sum_{n=-N}^N\sqrt{1+n^2} |c_n|^2\right)^{\frac{1}{2}}\leq C  \left(\frac{R}{\rho}\right)^{N\cdot \frac{\beta}{1+\beta}}M,
	\]
	where $\beta=\frac{C}{\delta}$ for some constant $C$.
	
	%\[
	%\| h|_{\Omega}-u_N|_{\Omega}  \|_{L^2(\Omega)}\leqslant \varepsilon \|h\|_{H^1(\tilde{\Omega})}+\frac{C}{R^N}\|u_N\|_{L^2(\partial O_R)},\quad \text{and}\quad \|u_N\|_{H^{1/2}(\partial O_R)}\leq C \varepsilon^{-\beta}\|h|_{\Omega}\|_{L^2(\Omega)}.
	%\]
\end{theorem}

\begin{proof}
	By the quantitative Runge's approximation in Theorem \ref{runge-approx-new-exterior}, for $\forall \varepsilon>0$ there exists $\{C_n\}_{n\in\mathbb{Z}}$,
    \[
    u_\varepsilon =\sum_{n\in\mathbb{Z}}C_nH_n^{(1)}(kr)e^{\mathrm{i}n\theta}
    \]
satisfying Helmholtz equation in $\mathbb{R}^2\setminus \overline{O}_\rho$  such that 
	\[
	\| u-u_\varepsilon \|_{L^2(\partial \Omega_\delta)}\leq \varepsilon  M, \quad \text{and}\quad \|u_\varepsilon\|_{H^{1/2}(\partial O_\rho)}\leq C \varepsilon^{-\beta}M,
	\]
where $\beta=\frac{C}{\delta}$ for some constant $C$. 	Let $c_n=C_nH_n^{(1)}(k\rho)$, the series read
   \[
    u_\varepsilon=\sum_{n\in\mathbb{Z}}c_n\frac{H_n^{(1)}(kr)}{H_n^{(1)}(k\rho)}e^{\mathrm{i}n\theta},
    \]
   which converges in $\mathbb{R}^2\setminus\overline{O}_\rho$.  Take 
	\[
	u_N(x)=\sum_{n=-N}^{N} c_n\frac{H_n^{(1)}(kr)}{H_n^{(1)}(k\rho)}e^{\mathrm{i}n\theta},
	\]
	then for $\rho<R$ and $O_R\Subset\Omega_\delta$, 
	\begin{align}
		\|u_\varepsilon-u_N\|_{L^2(\partial \Omega_\delta )} & \leq C\|u_\varepsilon-u_N\|_{L^\infty(O_R )}\leq C \sum_{|n|=N+1}^\infty |c_n\frac{H_n^{(1)}(k R)}{H_n^{(1)}(k\rho)}| \notag \\
		& \leq C \left(\sum_{|n|=N+1}^\infty |c_n|^2 \sum_{|n|=N+1}^\infty \frac{\rho^{2n}}{R^{2n}}  \right)^{1/2}\leq C \|u_\varepsilon\|_{L^2(\partial O_\rho)}\cdot \left(\frac{\rho}{R}\right)^N.\label{approx-uN}
	\end{align}
	Meanwhile,
	%\[
	%\|u_\varepsilon\|_{L^2(\partial O_\rho)}\leq \|u_\varepsilon\|_{H^{1/2}(\partial O_\rho)}
	%\leq C \varepsilon^{-\beta}M, 
	%\]
	%and similarly
	\begin{equation}\label{norm-uN}
	\|u_N\|_{L^2 (\partial O_\rho)}\leq \|u_\varepsilon\|_{L^2 (\partial O_\rho)}\leq C \varepsilon^{-\beta} M.
	\end{equation}
	As a result,
	\begin{align*}
		\| u-u_N|_{\partial\Omega_\delta}  \|_{L^2(\partial\Omega)} & \leqslant \|u-u_\varepsilon \|_{L^2(\partial\Omega_\delta)}+ \|u_\varepsilon-u_N \|_{L^2(\partial\Omega_\delta)}\\
		& \leq \varepsilon M+C\left(\frac{\rho}{R}\right)^N\varepsilon^{-\beta}M.
	\end{align*}
	Due to the arbitrariness of $\varepsilon$, taking $\varepsilon \sim (\rho/R)^{N\cdot \frac{1}{1+\beta}} $ yields
	\[
	\| u-u_N  \|_{L^2(\partial\Omega_\delta)}\leq C \left(\frac{\rho}{R}\right)^{N\cdot \frac{1}{1+\beta}}M.
	\]
	Substituting $\beta=\frac{C}{\delta}$ gives the result.
\end{proof}

The following proposition states that, for the analytic boundary case, the collocation method gives numerical results with exponential convergence.

\begin{prop}\label{num-ana}
    Follow the assumptions and notations in Theorem \ref{runge-part}. Further assume that $\partial\Omega$ is analytic parameterized by an analytic function $r(\theta)$, and $\partial\Omega_{\delta}$ is parameterized by $r_\delta(\theta)=r(\theta)+\delta$ (in this setting $dist(\Omega,\partial \Omega_\delta)\geq C\delta$ for some $C>0$).    
    Let $g(\theta)=u|_{\partial\Omega_\delta}$, $g_j=g(\theta_j)$, $\theta_j=\frac{2\pi j}{J}, j=0,\cdots,J-1$. Let $\mbf{c}_*=\{c^*_n\}_{n=-N}^N$ be the minimizer to
    \[
    F(\mbf{c})=\frac{1}{J}\sum_{j=0}^{J-1}|g_j-u_{(\mbf{c})}(\theta_j)|^2+\alpha \sum_{n=-N}^N|c_n|^2, 
    \]
    where $u_{(\mbf{c})}$ is the finite sum as \eqref{num-series}. Then
    \[
    \|u_{(c_*)}-u\|_{L^2(\partial\Omega_\delta)}\leq C\left(\left(\frac{\rho}{R}\right)^{N}+e^{-aJ}\right)^\kappa M,
    \]
    provided that the regularization parameter $\alpha\sim \left(\frac{\rho}{R}\right)^{2N}+e^{-2aJ} $. $C,a>0$ are constants depending on $\Omega$. $\kappa=\frac{1}{1+\beta}$ with $\beta\lesssim\frac{1}{\delta}$.
\end{prop} 

\begin{proof}

By the quantitative Runge approximation, for $\forall \varepsilon>0$, there exists a solution $u_\varepsilon$ that satisfies Helmholtz equation in $\mathbb{R}^2\setminus\overline{O}_\rho$ such that,
\[
\|u-u_\varepsilon\|_{L^2(\partial\Omega_{\delta/2})}\leq \varepsilon M, \quad \text{and} \quad \|u_\varepsilon\|_{H^{1/2}(\partial O_\rho)}\leq C \varepsilon^{-\beta}M=:\mathcal{M}, 
\]
where $\beta\lesssim\frac{1}{\delta}$ and $\partial\Omega_{\delta/2}$ is parameterized by $r(\theta)+\delta/2$. 

For $u_\varepsilon$, by \eqref{approx-uN} and \eqref{norm-uN}, there exists $\{\tilde{c}_n\}_{n=-N}^N$ such that
\begin{equation}%\label{num-series}
	u_N(x)=\sum_{n=-N}^N \tilde{c}_n \frac{H_n^{(1)} (kr)}{H_n^{(1)} (k\rho)}e^{\mathrm{i}n\theta},
\end{equation}
converges to $u_\varepsilon$ with
\[
\|u_\varepsilon-u_N\|_{L^2(\partial\Omega_{\delta/2})}\lesssim \left(\frac{\rho}{R}\right)^{N}\mathcal{M},\quad \left(\sum_{n=-N}^N |\tilde{c}_n|^2\right)^{\frac{1}{2}}\asymp\|u_N\|_{L^2(\partial O_\rho)}\leq\mathcal{M}. 
\]
Denote $r_j=r(\theta_j)+\delta$. Denote $\eta^2=\left(\frac{\rho}{R}\right)^{2N}+e^{-2aJ}$ and take $\varepsilon=\eta^{\frac{1}{1+\beta}}$, which gives $\mathcal{M}\lesssim\eta^{-\frac{\beta}{1+\beta}}M$. Combining interior regularity we have
\begin{align}
&\frac{1}{J}\sum_{j=0}^{J-1}|g_j-u_N(r_j,\theta_j)|^2  \leq C\|u-u_N\|^2_{L^2(\partial\Omega_{\delta/2})}\notag \\
\leq & C \left( \|u-u_\varepsilon\|^2_{L^2(\partial\Omega_{\delta/2})}+\|u_\varepsilon-u_N\|^2_{L^2(\partial\Omega_{\delta/2})}\right)\notag\\
\leq & C\left(\varepsilon^2M^2+\left(\frac{\rho}{R}\right)^{2N}\mathcal{M}^2\right)\leq C\left(\varepsilon^2M^2+\eta^2\varepsilon^{-2\beta}M^2  \right)=C \eta^{\frac{2}{1+\beta}}M^2\label{u-uN}.
\end{align}
By definition of the minimizer and choice of $\alpha\sim \eta^2$, there is
    \begin{align}
    \frac{1}{J}\sum_{j=0}^{J-1}|g_j-u_{(c_*)}(r_j,\theta_j)|^2 & \leq \frac{1}{J}\sum_{j=0}^{J-1}|g_j-u_N(r_j,\theta_j)|^2+\alpha \sum_{n=-N}^N |\tilde{c}_n|^2\notag \\
    &\lesssim  \eta^{\frac{2}{1+\beta}}M^2+\alpha \varepsilon^{-2\beta}M^2 
    \lesssim \eta^{\frac{2}{1+\beta}}M^2.\label{residual}
    \end{align}
    Meanwhile, 
    \begin{equation}\label{bound-c}
    \sum_{n=-N}^N|c_n^*|^2\leq \frac{1}{\alpha}\cdot\frac{1}{J}\sum_{j=0}^{J-1}|g_j-u_N(r_j,\theta_j)|^2+ \sum_{n=-N}^N |\tilde{c}_n|^2\lesssim \eta^{-\frac{2\beta}{1+\beta}} M^2\lesssim \mathcal{M}^2.
    \end{equation}

 Complexify the function 
    \[
    f(z)=u_{N}(r_\delta(z),z)-u_{(c^*)}(r_\delta(z),z), \quad z\in\mathbb{C}
    \]
    where $r_\delta(z)=r(z)+\delta$.
    Then $f(z)$ is holomorphic in a neighborhood of $[0,2\pi]\times\{0\}$ in $\mathbb{C}$, periodic in the real-axis direction, and
    \[ f|_{Im(z)=0} =u_{N}(r_\delta(\theta),\theta)-u_{(\mbf{c}_*)}(r_\delta(\theta),\theta). \]
    Moreover, by interior analyticity estimates \cite{Morrey1957}, there is \(|f|\leq C\mathcal{M} \) in the complex neighborhood, where $C$ depends on the analytic function of boundary parameterization, $dist(\partial\Omega_\delta,\partial O_\rho)$ and is independent of $u$.
    Thus, by the exponential convergence of the trapezoidal rule for periodic analytic functions (see, e.g., \cite{Trefethen2014}),
    \[
    \left|\frac{2\pi}{J}\sum_{j=0}^{J-1} |f(\theta_j)|^2-\|f|_{[0,2\pi]}\|_{L^2(0,2\pi)} \right|\leq C\mathcal{M} e^{-aJ},
    \]
    where $C$ and $a$ depend on $r(\theta)$. 
    
    As a result,
    \begin{align}
         \|u_{(c_*)}-u_{N}\|^2_{L^2(\partial\Omega_\delta)}&\leq \frac{2\pi}{J}\sum_{j=0}^{J-1}|u_N(r_j,\theta_j)-u_{(c_*)}(r_j,\theta_j)|^2+C\mathcal{M}^2e^{-2aJ}\notag\\
         &\leq \frac{C}{J}\sum_{j=0}^{J-1}|u_{(c_*)}(r_j,\theta_j)-g_j|^2+\frac{C}{J}\sum_{j=0}^{J-1}|g_j-u_N(r_j,\theta_j)|^2+C\mathcal{M}^2e^{-2aJ}\notag\\
         &\leq C\eta^{\frac{2}{1+\beta}}M^2+C\mathcal{M}^2\eta^2\lesssim \eta^{\frac{2}{1+\beta}}M^2.\label{L2-error-uN}
    \end{align}
    Combining \eqref{L2-error-uN} and \eqref{u-uN} yields
    \[
    \|u_{(c_*)}-u\|_{L^2(\partial\Omega_\delta)}\leq \|u_{(c_*)}-u_N\|_{L^2(\partial\Omega_\delta)}+\|u_N-u\|_{L^2(\partial\Omega_\delta)}\leq C\eta^{\frac{1}{1+\beta}}M,
    \]
    which completes the proof.

\end{proof}

\begin{remark}
    If $\partial\Omega_\delta$ is not analytic, one can get the numerical estimate similarly, while $e^{-aJ}$ here is weakened by $h$, the maximum distance between two adjacent collocation points. 
\end{remark}

By Graf's addition theorem, the cylindrical harmonics can be converted to fundamental solutions with poles distributed along a circle. Both kinds of basis can be taken as learning samples in the learning based numerical method \cite{Chen2024}. %Thus, Proposition \ref{num-ana} can be viewed as an extension of results in \cite{Barnett2008} and \cite{Chen2024} to general geometries.

\paragraph{Numerical validation}

We use a numerical example to illustrate the approximation result. The boundary of the domain $\Omega_\delta$ is parameterized by $(r_\delta(\theta)\cos\theta,r_\delta(\theta)\sin\theta)$ where
\[
r_\delta(\theta)=1+0.15\cos(6\theta),
\]
is analytic.

The exact solution is taken as 
\[
u=\sum_{i=1}^3\Phi(x_i,y_i),\quad (x_i,y_i)=(\cos(2(i-1)\pi/3),\sin(2(i-1)\pi/3)),
\]
with three point sources placed inside the flower domain. $\Phi(x,y)$ is the fundamental solution corresponding to $k=6$. 
Then it satisfies the Helmholtz equation in $\mathbb{R}^2\setminus\overline{\Omega}$ and the radiation condition. Here, $\partial\Omega$ is parameterized by $r(\theta)=0.85+0.15\cos(6\theta)$. 

In numerical computation, we set $N=60$, $J=360$, $\rho=0.45$. with the algorithm in Proposition \ref{num-ana}. 
Fig~\ref{fig-num-eg} presents the exact solution, the approximation solution, and the error distribution. It can be seen that, in such an irregular domain, the exterior solution can be approximated by cylindrical harmonics. The numerical approximant satisfies the equation in a larger domain. Inside the computational domain ($\mathbb{R}^2\setminus\Omega_\delta$) the numerical result shows high approximation accuracy, while it grows significantly in $\mathbb{R}^2\setminus\overline{\Omega}$ (Fig~\ref{fig:eg-approx}).

For comparison, when the source is placed further to the boundary (Fig~\ref{fig:eg-b-exact}), with 
\[
u_b=\sum_{i=1}^3\Phi(x_{b,i},y_{b,i}),\quad (x_{b,i},y_{b,i})=(0.5\cos(2(i-1)\pi/3),0.5\sin(2(i-1)\pi/3)),
\]
and other settings are fixed, the numerical error is significantly reduced to an order of $10^{-16}$ (Fig~\ref{fig:eg-b-error}), which is consistent with the estimate that $\beta$ gets lower as $\delta$ increases.

\begin{figure}[H]
   % \centering
    \subfigure[Exact solution $Re(u)$]{
        \includegraphics[width=0.31\linewidth]{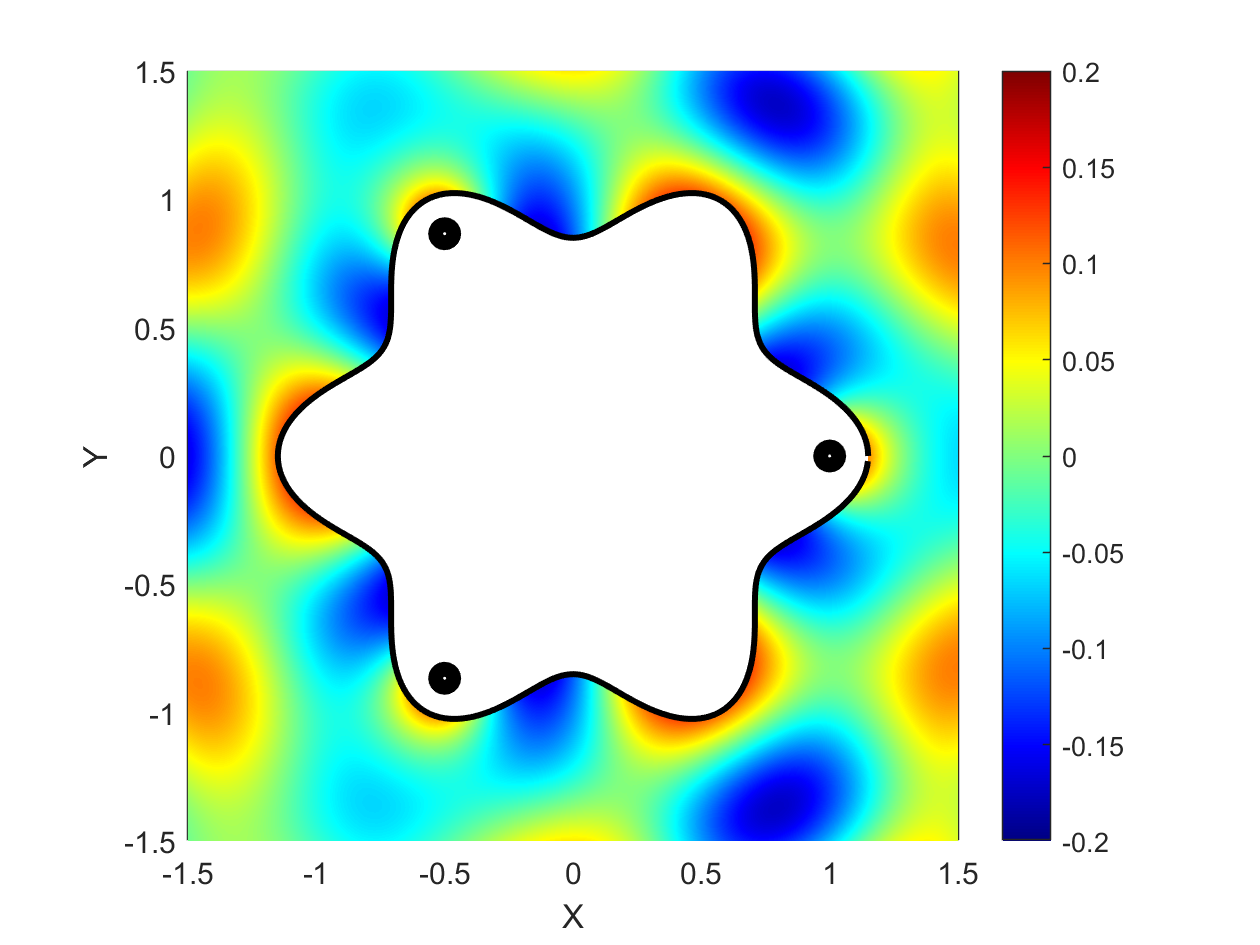}
        \label{fig:eg-exact}
    }
    \subfigure[Approximant]{
        \includegraphics[width=0.31\linewidth]{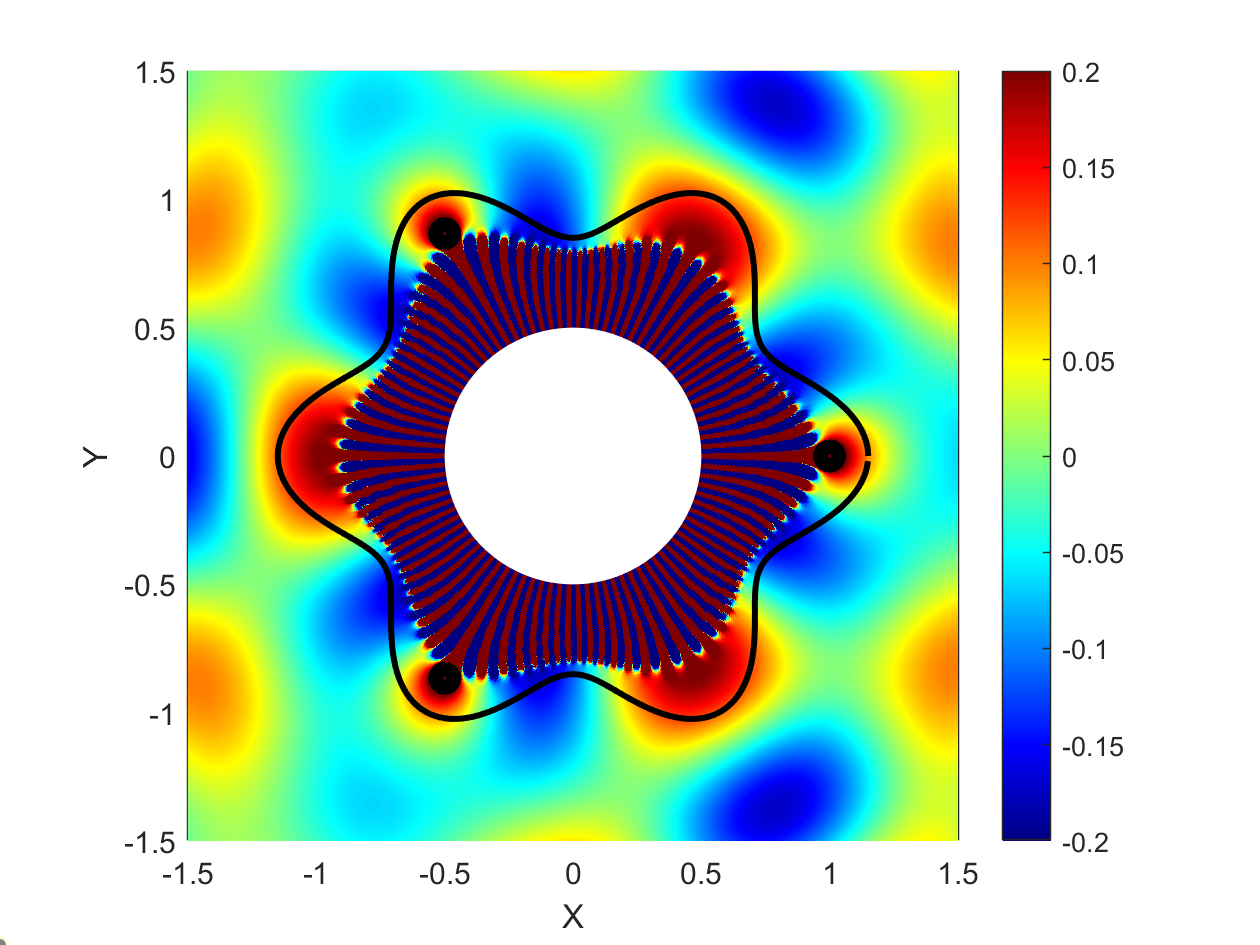}
        \label{fig:eg-approx}
    }
    \subfigure[Error distribution]{
        \includegraphics[width=0.31\linewidth]{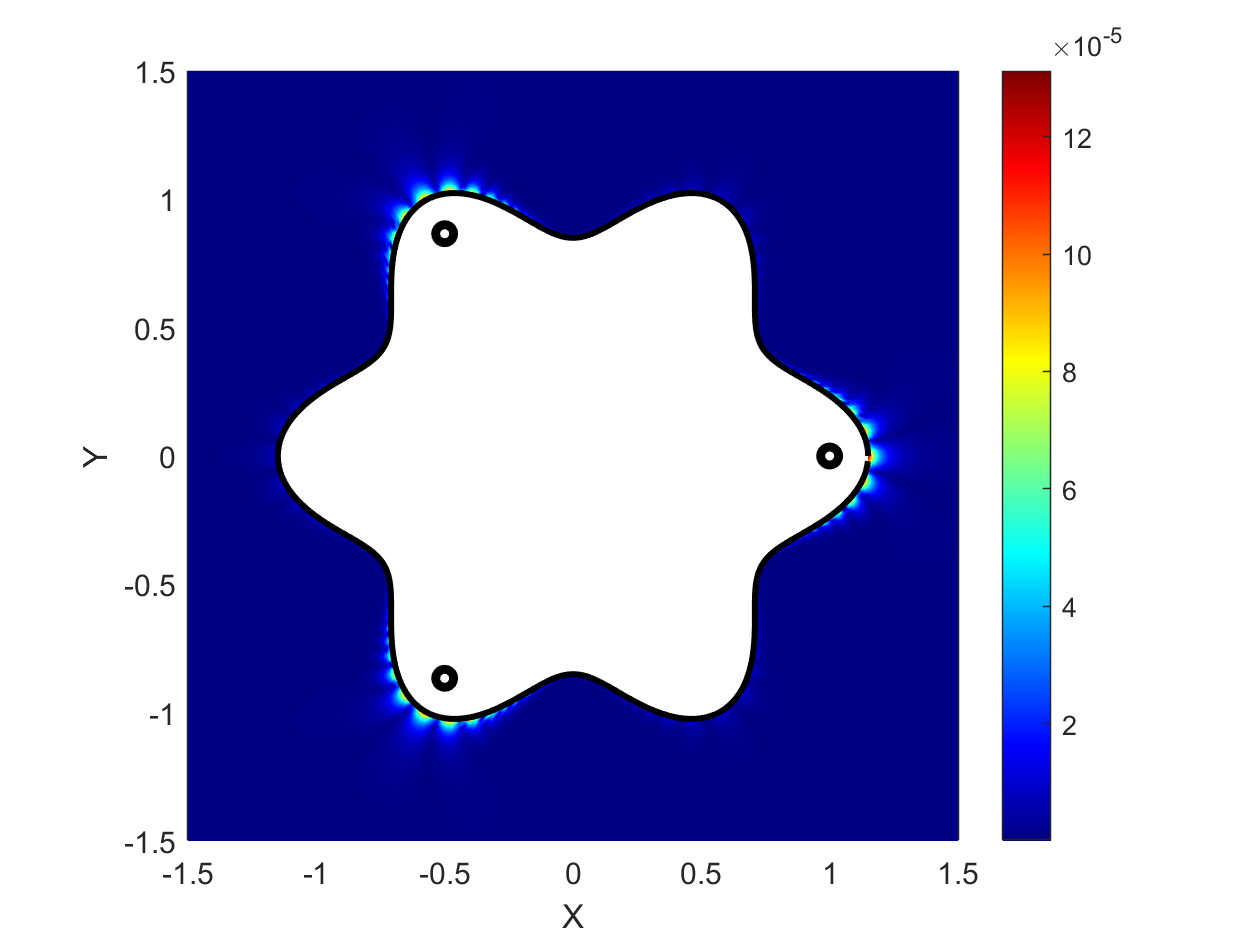}
        \label{fig:eg-error}
    }
    \caption{Comparison of the exact solution $u$, the approximated solution, and the pointwise error distribution.}
    \label{fig-num-eg}
\end{figure}

\begin{figure}[H]
   \begin{center}% \centering
    \subfigure[Exact solution $Re(u_b)$]{
        \includegraphics[width=0.31\linewidth]{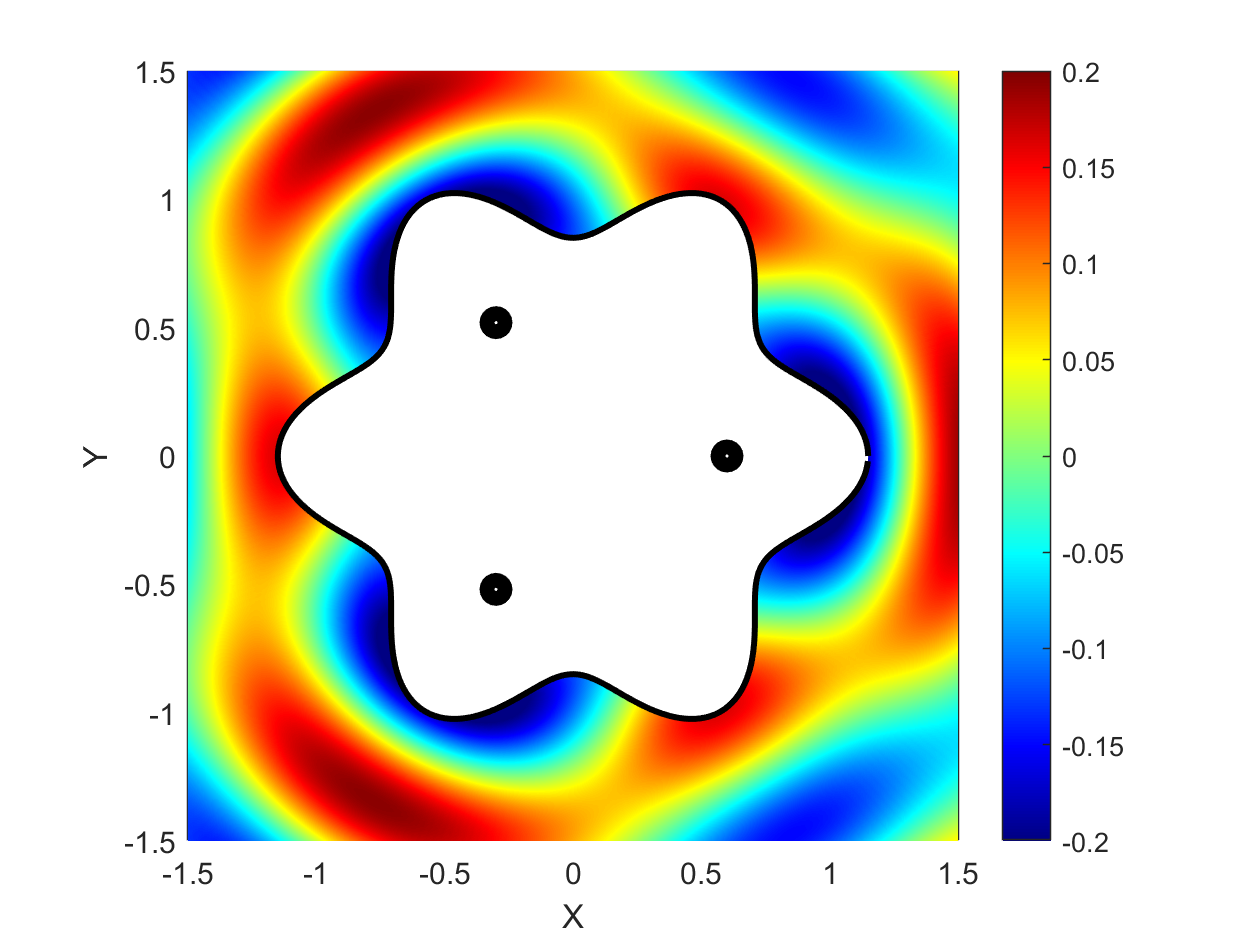}
        \label{fig:eg-b-exact}
    }
    \subfigure[Approximaant]{
        \includegraphics[width=0.31\linewidth]{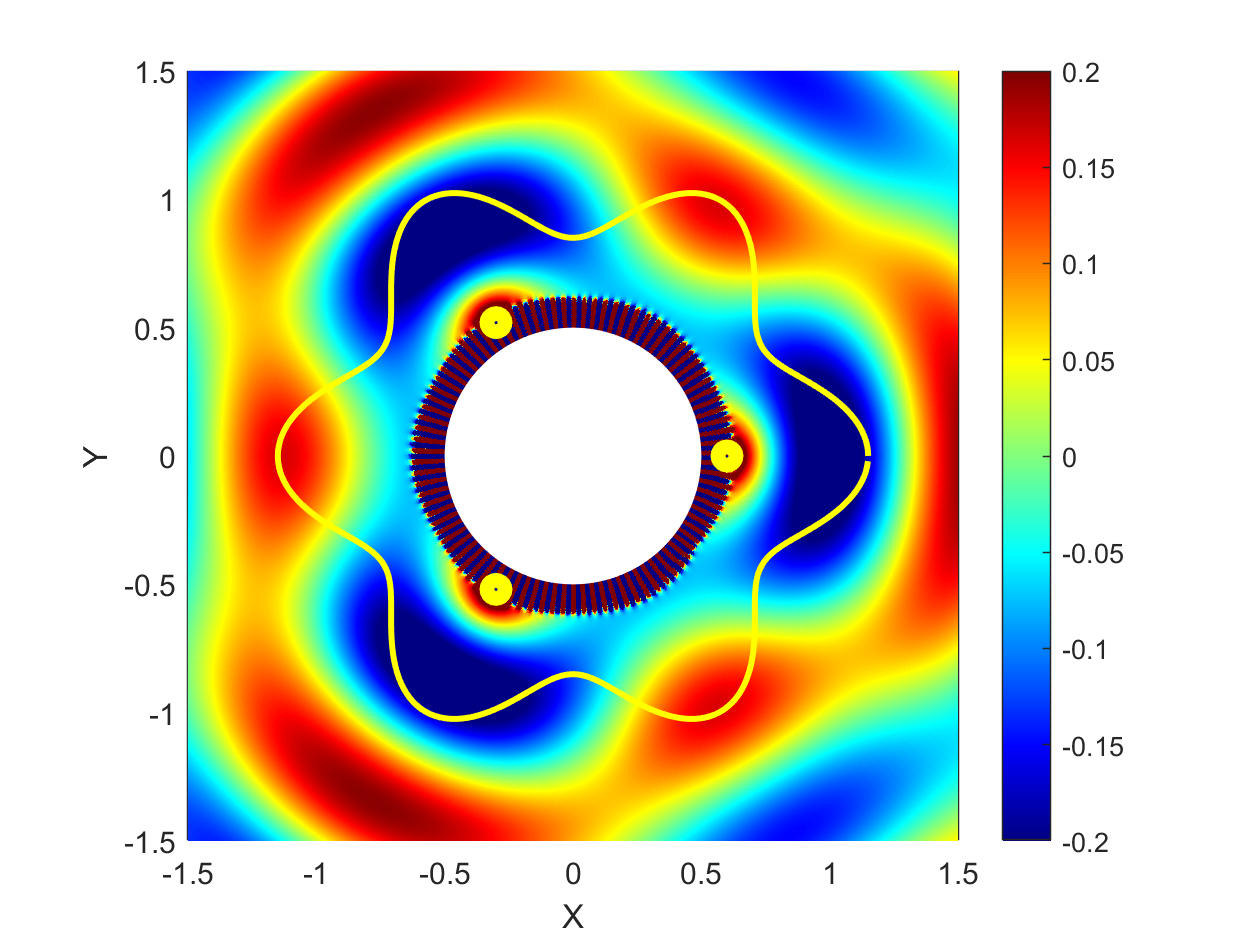}
        \label{fig:eg-b-approx}
    }
    \subfigure[Error distribution]{
        \includegraphics[width=0.31\linewidth]{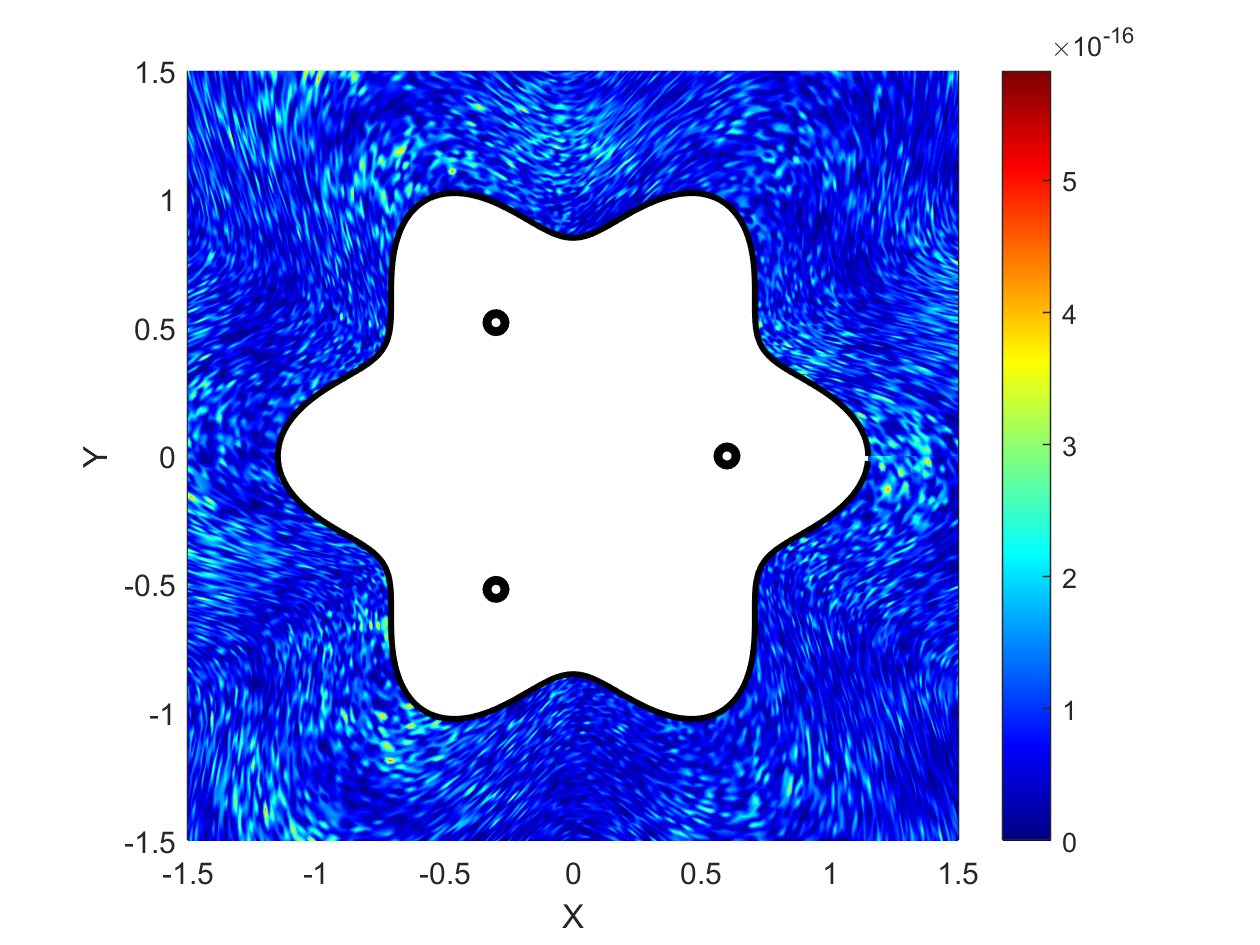}
        \label{fig:eg-b-error}
    }
    \caption{Comparison of the exact solution $u_b$, the approximated solution, and the pointwise error distribution.}
    \label{fig-num-eg-b}
    \end{center}
\end{figure}

\iffalse
\begin{figure}[H]
    \centering
    \includegraphics[width=0.32\linewidth]{eg1.png}
    \includegraphics[width=0.32\linewidth]{eg2.png}
    \includegraphics[width=0.32\linewidth]{eg4.png}
    \caption{The exact solution, approximant, and error distribution respectively.}
    \label{fig-num-eg}
\end{figure}

\begin{figure}[H]
    \centering
    \includegraphics[width=0.32\linewidth]{eg1-b.png}
    \includegraphics[width=0.32\linewidth]{eg2-b.png}
    \includegraphics[width=0.32\linewidth]{eg4-b.png}
    \caption{The exact solution $u_b$, approximant, and error distribution respectively.}
    \label{fig-num-eg-b}
\end{figure}
\fi
%The result can be extended to the case of an exterior problem which is often raised in scattering problems. It will be discussed in our coming paper.

\section{Concluding remark}
In this work, a quantitative Runge approximation for Helmholtz equation is provided, represented by cylindrical harmonics. The index depends explicitly on the distance to which the solutions can be extended. The estimate can be used in the numerical approximation of solutions to Helmholtz equations with spectral accuracy.

The related results can be extended to the three dimensional case in a similar way, by using spherical Hankel functions and spherical harmonic basis. Similar results hold for Laplace equation. 

According to the idea of the learning based method for Helmholtz equation with high frequency \cite{Chen2024}, the present result provides a choice of learning samples. Moreover, if the solution cannot be extended, we can decompose the singularities, lift the regularities, and further using the present approximation result, which will be discussed in our forthcoming paper.

\section*{Acknowledgments}
This work was supported by China's National Key Research and Development Programs (No. 2024YFA1012401), China's National Natural Science Foundation (Nos. 12241103, 12571458), Program for Innovative Research Team of Shanghai University of Finance and Economics (No. 2022110917), the Fundamental Research Funds for the Central Universities (No. 2025110603), and funds from Laboratory of Mathematics for Nonlinear Science, Fudan University.

\appendix
\section{Appendix}

\subsection*{Proof of Proposition \ref{runge-ball0}} %
    \begin{proof}

Denote $g=\gamma u|_{\partial O_R}$ and $f(\theta), \theta\in[0,2\pi]$ its parameterization. 
Since $\{\mathrm{e}^{\mathrm{i}n\theta}\}_{n\in\mathbb{Z}}$ forms a complete basis for $L^2(0,2\pi)$, then for $f\in L^2(0,2\pi)$, $f(\theta)=\sum_{m\in \mathbb{Z}}\hat{f}_m e^{\mathrm{i}m\theta}$.  Therefore, by uniqueness of the boundary value problem,
\[
u(\mbf{x})=\sum_{m\in\mathbb{Z}}C_mJ_m(kr)e^{\mathrm{i}m\theta}, \quad C_m=\frac{\hat{f}_m}{J_m(kR)},\quad \mbf{x}\in O_R.
\]
 Define 
        \[
        \tilde{u}_N(r,\theta )=\sum_{|m|\leq N} \frac{\hat{f}_m}{J_m(k R)}J_m(kr)e^{\mathrm{i}n\theta}.
        \]
%Without loss of generality, we assume that $R=1$ in the following.  
For $\varepsilon>0$, take $N\geq \frac{1}{\varepsilon^2}$, then
        \[
        \|u-\tilde{u}_N\|^2_{L^2(\partial O_R)}=2\pi \sum_{|m|>N}\frac{1}{\sqrt{1+\frac{m^2}{R^2}}}\sqrt{1+\frac{m^2}{R^2}}|\hat{f}_m|^2\leq \frac{2\pi R}{N}\sum_{m\in\mathbb{Z}}\sqrt{1+m^2}|\hat{f}_m|^2\leq \varepsilon^2 \|g\|^2_{H^{1/2}(\partial O_R)},
        \]
        i.e., $\|u-\tilde{u}_N|_{\partial O_R}\|_{L^2(\partial O_R)}\leq \varepsilon \|u\|_{H^{1/2}(\partial O_R)}$. 
      Then the norm of $\tilde{u}$ on $\partial O_\rho$ with $\rho> 1$ has the estimate as, 
        \[
        \sum_{|m|\leq N}\sqrt{1+\frac{m^2}{\rho^2}}\left|\hat{f}_m\frac{J_m(k\rho)}{J_m(kR)}\right|^2\leq C N \left(\frac{\rho}{R}\right)^{2N}\sum_{|m|\leq N}|\hat{f}_m|^2,
        \]
        where we have used $J_n(x)\sim \frac{1}{\sqrt{2\pi n}}\left(\frac{ex}{2n}\right)^n, n\rightarrow \infty $ uniformly on compact subsets of $(0,+\infty)$.
        As a result,
        \[
        \|\tilde{u}_N\|_{H^{1/2}(\partial O_\rho)}\leq C\sqrt{N}e^{N\log \frac{\rho}{R}}\|g\|_{L^2(\partial O_R)}\leq C e^{\left(\log\frac{\rho}{R}\right) \varepsilon^{-2}}\|u\|_{L^2(\partial O_R)}.
        \]

        For case ii), if $u$ is a solution to the boundary value problem in $O_\sigma$, $R<\sigma<\rho$, then we have proved that
        \[
        u(x)=\sum_{m\in\mathbb{Z}}C_mJ_m(kr)e^{\mathrm{i}m\theta},\quad x\in O_\sigma.
        \]
        Let $g=\gamma u|_{\partial O_\sigma}$, then $\hat{g}(m)=C_mJ_m(k\sigma)$. Still, define
        \[
        \tilde{u}_N(x)=\sum_{|m|\leqslant N}C_n J_m(kr)e^{\mathrm{i}m\theta}, 
        \]
        which satisfies the equation in $O_\rho$ and the radiation condition. Similarly,
        \begin{align*}
        \|u-\tilde{u}_N|_{\partial O_R}\|_{L^2(\partial O_R)}^2 & =2\pi \sum_{|m|>N}\frac{1}{\sqrt{1+\frac{m^2}{R^2}}}\sqrt{1+\frac{m^2}{R^2}}\left|\frac{\hat{f}_m}{J_m(k\sigma)}J_m(kR)\right|^2\\
        &\leq  C\frac{1}{N}\left(\frac{R}{\sigma}\right)^{2N}\sum_{m\in \mathbb{Z}}\sqrt{1+\frac{m^2}{R^2}}|\hat{f}_m|^2\leq \varepsilon^2\|g\|^2_{H^{1/2}(\partial O_R)}
        \end{align*}
        Take sufficiently large $N$ that $C\left(\frac{R}{\sigma}\right)^N\leq \varepsilon$, e.g., $N=\left[2\log\frac{\varepsilon}{C}/\log\frac{R}{\sigma}\right]+1$, and $\|u-\tilde{u}_N|_{\partial O_R}\|_{L^2(\partial O_R)}\leq \varepsilon \|g\|_{H^{1/2}(\partial O_R)}$ holds.
        Then, 
        \begin{align*}
        &\sum_{|m|\leqslant N}\sqrt{1+\frac{m^2}{\rho^2}}\left|\hat{f}_m\frac{J_m(k\rho)}{J_m(k\sigma)}\right|^2 =\sum_{|m|\leqslant N}\sqrt{1+\frac{m^2}{\rho^2}}|\hat{f}_m|^2\left(\frac{J_m(k\rho)}{J_m(kR)}\frac{J_m(kR)}{J_m(k\sigma)}\right)^2\\
        & \leq CN\left(\frac{\rho}{R}\right)^{2N}\sum_{|m|\leqslant N}|\hat{f}_m|^2\left(\frac{J_m(kR)}{J_m(k\sigma)}\right)^2
        \leq C e^{2N\log\left(\frac{\rho}{R}\right)}\|u\|^2_{L^2(\partial O_R)}\\
        &\leq C \exp\left( 2\log\frac{\varepsilon}{C}\cdot\frac{\log\frac{\rho}{R}}{\ln \frac{R}{\sigma}}\right)=C\varepsilon^{-2\frac{\log\frac{\rho}{R}}{\log \frac{\sigma}{R}}}\|u\|^2_{L^2(\partial O_R)}.
        \end{align*}
        
         \end{proof}

\begin{lemma}\label{estimate-u12}
    Assume that $\delta>0$ is a constant, $l$ is a smooth arc segment such that $l\subset O_{\delta/2}$. $\Gamma=\partial O_{\delta}$. Define 
    $u(x)=\int_{l}H_0^{(1)}(k|x-y|)\mu(y)\mathrm{d}y$, where $\mu\in L^2(l)$. Then there is 
\[
\|u\|_{H^{1/2}(\partial O_\delta)}\leq C\sqrt{\delta}\|\mu\|_{L^2(l)},
\]
where $C>0$ is a constant independent of $\delta$ and $\mu$. 
\end{lemma}
\begin{proof}
 
For $0<s<1$, $u\in L^2(\Gamma)$ where $\Gamma$ is a closed smooth curve,
\[
\|u\|^2_{H^s(\Gamma)}=\|u\|^2_{L^2(\Gamma)}+|u|^2_{H^s(\Gamma)},
\]
where 
\[
|u|^2_{H^s(\Gamma)}=\int_\Gamma\int_\Gamma \frac{|u(x)-u(y)|^2}{|x-y|^{1+2s}}\mathrm{d}s_x\mathrm{d}s_y.
\]
For
\[
u(x)=\int_l \Phi(x,z)\mu(z)\mathrm{d}z, \quad \Phi(x,z)=\frac{\mathrm{i}}{4}H_0^{(1)}(k|x-z|),
\]
there is
\[
|u(x)-u(y)|^2\leq \left(\int_{l}|\Phi(x,z)-\Phi(y,z)|^2\mathrm{d}z\right)\|\mu\|^2_{L^2(l)}.
\]
%Let $\Phi(x,z)=H_0^{(1)}(k|x-z|)$.
Then for $s=1/2$, we have
\[
|u(x)|^2_{H^{1/2}(\Gamma)}\leq \|\mu\|^2_{L^2(l)}\int_{\Gamma}\int_\Gamma\frac{\int_l|\Phi(x,z)-\Phi(y,z)|^2\mathrm{d}z}{|x-y|^2}\mathrm{d}s_x\mathrm{d}s_y,
\]
in which
\[
|\Phi(x,z)-\Phi(y,z)|\leq |x-y|\sup_{\xi\in \Gamma}|\nabla_x\Phi(\xi,z)|\leq C|x-y|\sup_{\xi\in\Gamma}\frac{1}{|\xi-z|}\leq \frac{C}{d}|x-y|.
\]
with $d=dist(l,\Gamma)$. 
Therefore,
\[
|u|^2_{H^{1/2}(\Gamma)}\leq \|\mu\|^2_{L^2(l)}|\Gamma|^2\frac{|l|}{d^2}.%\leq C\delta\|\mu\|^2_{L^2(l)}.
\]
For $\Gamma =\partial O_{\delta}$, $l\subset O_{\delta/2}$, $d\geq \delta$,
\[
|u|^2_{H^{1/2}(\partial O_{\delta/2})}\leq C\delta \|\mu\|^2_{L^2(l)}.
\]
Consider the $L^2$ term, since $H_0^{(1)}(|x|)\sim C\ln |x|$ as $|x|\rightarrow 0$, then
\[
\|u\|^2_{L^{2}(\partial O_{\delta/2})}\leq C\delta^2\ln^2\delta\|\mu\|_{L^2(l)}.
\]
We finally get
\[
\|u\|_{H^{1/2}(\partial O_{\delta/2})}\leq C\sqrt{\delta}\|\mu\|_{L^2(l)}.
\]
   
\end{proof}
%where $C$ is independent of $\delta$.

\addcontentsline{toc}{section}{

\end{document}